\documentclass[12pt]{article}
\usepackage{array,amsmath}
\usepackage{amssymb}
\usepackage{graphicx,subfigure}

\makeatletter
 \usepackage{fullpage}
 \usepackage{amsthm}

\usepackage{fullpage}
\makeatother

\begin{document}

\newtheorem{thm}{Theorem}[section]
\newtheorem{lem}[thm]{Lemma}
\newtheorem{con}[thm]{Conjecture}
\newtheorem{cons}[thm]{Construction}
\newtheorem{prop}[thm]{Proposition}
\newtheorem{cor}[thm]{Corollary}
\newtheorem{claim}[thm]{Claim}
\newtheorem{obs}[thm]{Observation}
\newtheorem{que}[thm]{Question}
\newtheorem{defn}[thm]{Definition}
\newtheorem{example}[thm]{Example}
\newcommand{\di}{\displaystyle}
\def\dfc{\mathrm{def}}
\def\cF{{\cal F}}
\def\cH{{\cal H}}
\def\cK{{\cal K}}
\def\cM{{\cal M}}
\def\cA{{\cal A}}
\def\cB{{\cal B}}
\def\cG{{\cal G}}
\def\cP{{\cal P}}
\def\cC{{\cal C}}
\def\ap{\alpha'}
\def\Frk{F_k^{2r+1}}
\def\nul{\varnothing} 
\def\st{\colon\,}  
\def\MAP#1#2#3{#1\colon\,#2\to#3}
\def\VEC#1#2#3{#1_{#2},\ldots,#1_{#3}}
\def\VECOP#1#2#3#4{#1_{#2}#4\cdots #4 #1_{#3}}
\def\SE#1#2#3{\sum_{#1=#2}^{#3}}  \def\SGE#1#2{\sum_{#1\ge#2}}
\def\PE#1#2#3{\prod_{#1=#2}^{#3}} \def\PGE#1#2{\prod_{#1\ge#2}}
\def\UE#1#2#3{\bigcup_{#1=#2}^{#3}}
\def\FR#1#2{\frac{#1}{#2}}
\def\FL#1{\left\lfloor{#1}\right\rfloor} 
\def\CL#1{\left\lceil{#1}\right\rceil}

\title{The second largest eigenvalue and vertex-connectivity of regular multigraphs}
\author{
  Suil O\thanks{
  Applied Algebra and Optimization Research Center,
  Sungkyunkwan University,
  Suwon, Gyeonggi-do~~16419,
  {\tt suilo@skku.edu}}
}

\maketitle

\begin{abstract}
Let $\mu_2(G)$ be the second smallest Laplacian eigenvalue of a graph $G$.
The vertex-connectivity of $G$, written $\kappa(G)$, is the minimum size of a vertex set $S$ such that $G-S$ is
disconnected.  Fiedler proved that $\mu_2(G) \le \kappa(G)$ for a non-complete
simple graph $G$; for this reason $\mu_2(G)$ is called the ``algebraic
connectivity'' of $G$. His result can be extended to multigraphs: 
for any multigraph $G$ whose underlying graph is not a complete graph, we have $\mu_2(G) \le \kappa(G) m(G)$,
where for a pair of vertices $u$ and $v$, let $m(u,v)$ be the number of edges with endpoints $u$
and $v$ and $m(G)=\max_{(u,v) \in E(G)} m(u,v)$.


Let $\lambda_2(G)$ be the second largest eigenvalue of a graph $G$.
We also prove that for any $d$-regular multigraph $G$ whose underlying graph is not the complete graph with 2 vertices,
if $\lambda_2(G) < \frac 34d$, then $G$ is 2-connected.
For $t\ge2$ and infinitely many $d$, we construct $d$-regular multigraphs $H$
with $\lambda_2(H)=0$ (or $\mu_2(H)=d$), $\kappa(H)=t$, and $m(H)=\frac dt$.  These graphs show that
the inequality $\mu_2(G) \le \kappa(G) m(G)$ is sharp.  
Furthermore, the existence of them 
implies that
there is no upper bound for $\lambda_2(G)$ in a $d$-regular multigraph $G$ to guarantee a certain vertex-connectivity greater than equal to 3.
\end{abstract}

\smallskip



\section {Preliminaries}
A {\it simple} graph is a graph without loops or multiple edges, and a {\it multigraph} is a graph that may have multiple edges but does not contain loops. 
The {\it adjacency matrix} $A(G)$ of $G$ is the $n$-by-$n$ matrix in which the entry $a_{i,j}$ is the number of edges in $G$ with endpoints $\{v_i, v_j\}$,
where $V(G)=\{v_1, ... , v_n\}$. The {\it eigenvalues} of $G$ are the eigenvalues of its adjacency matrix $A(G)$.
Let $\lambda_i(G)$ be the $i$-th largest eigenvalue of $G$.
The {\it Laplacian matrix} of $G$ is $D(G)-A(G)$, where $D(G)$ is the diagonal matrix of vertex degrees.
Let $\mu_i(G)$ be the $i$-th smallest Laplacian eigenvalue of $G$. Note that $\mu_1(G)=0$ for any graph $G$ . Also if $G$ is $d$-regular, then $\lambda_1(G)=d$, and $\lambda_i(G)=d-\mu_i(G)$, thus $\lambda_2(G)$ and $\mu_2(G)$ are directly related.

A multigraph $G$
is {\it $k$-vertex-connected} if $G$ has more than $k$ vertices and every subgraph obtained by deleting fewer than $k$ vertices is connected;
the {\it vertex-connectivity} of $G$, written $\kappa(G)$, is the maximum $k$ such that $G$ is $k$-vertex-connected.
 Fiedler~\cite{F} proved that 
\begin{equation}\label{fiineq} 
 \mu_2(G) \le \kappa(G)
 \end{equation}
for a non-complete simple graph $G$; for this reason $\mu_2(G)$ is called the ``algebraic connectivity'' of $G$. A lot of research in graph theory over the last 40 years was stimulated by Fiedler's work. 
We can also easily extend his work to multigraphs with the same idea of his proof for inequality~(\ref{fiineq}). 
For a pair of vertices $u$ and $v$, let $m(u,v)$ be the number of edges between $u$ and $v$. 
Let $m(G)=\max_{(u,v) \in E(G)} m(u,v)$.
We call $m(G)$ the multiplicity of the multigraph $G$.
Theorem~\ref{fiedlermultiversion} is an extension of Fiedler's result for multigraphs.
\begin{thm}\label{fiedlermultiversion}
For any multigraph $G$ whose underlying graph is not a complete graph, we have 
\begin{equation}\label{ineq1}
\mu_2(G) \le \kappa(G)m(G).
\end{equation}
\end{thm}
Note that Theorem~\ref{fiedlermultiversion} implies Fielder's inequality~(\ref{fiineq}) by taking $m(G)=1$ (when $G$ is simple). 
In Section~\ref{consec}, we construct graphs showing the tightness of the bound in inequality~(\ref{ineq1}).
We can prove Theorem~\ref{fiedlermultiversion} by slightly modifying the proof of 3.3 in~\cite{F}.
Here is a rough proof. For a multigraph $G$, let $G_1$ be a multigraph obtained from $G$ by removing a vertex. 
Consider a new graph $G'$ obtained from $G_1$ and a single vertex $v$ by adding $m(G)$ edges between $v$ and each vertex in $V(G')$.
Then $L(G')$ has an eigenvector $(x~ 0 )^{T}$ corresponding to an eigenvalue $\mu_2(G_1)+m(G)$, 
where $x$ is an eigenvector of $L(G_1)$ corresponding to $\mu_2(G_1)$.
Thus $\mu_2(G') \le \mu_2(G_1)+m(G)$, which gives the desired result.

A multigraph $G$ is {\it $t$-edge-connected} if every subgraph obtained by deleting fewer than $t$ edges is connected;
the {\it edge-connectivity} of $G$, written $\kappa'(G)$, is the maximum $t$ such that $G$ is $t$-edge-connected. 
Since $\kappa'(G) \ge \kappa(G)$ for any multigraph $G$,
we have the following.

\begin{cor}
For any multigraph $G$ whose underlying graph is not a complete graph, we have 
\begin{equation}\label{kappa'ineq}
\mu_2(G) \le \kappa'(G)m(G).
\end{equation}

\end{cor}

We could not find any multigraph satisfying equality in inequality~(\ref{kappa'ineq}) when $m(G) \ge 2$. However,
there exist infinitely many multigraphs whose algebraic connectivity is bigger than their edge-connectivity. Of course, some of them are multigraphs whose underlying graphs are not complete graphs. For example, consider the graph $F$ obtained from the cycle on 4 vertices by duplicating $\frac {d-1}2$ times on two incident edges
and $\frac{d+1}{2}$ times on the other two incident edges.
Then $\mu_2(F)=\frac{3d}{2}- \frac{\sqrt{d^2+8}}{2}$ and  $\kappa'(F)=d-1$, which means $\mu_2(F) > \kappa'(F)$.
We believe that there is no big gap between $\mu_2$ and $\kappa'$ unlike between $\mu_2$ and $\kappa$.

\begin{que}
Is there a constant $c$ such that $\mu_2(G) \le \kappa'(G) + c$ for any multigraph $G$?
\end{que}


Fiedler's results have been improved and refined by various authors (See~\cite{C1},~\cite{C2},~\cite{D},~\cite{KMNS},~\cite{KS}).
The author~(\cite{O2}) extended some of these results to multigraphs. 
A {\it $d$-regular multigraph} $G$ is a multigraph such that
for any vertex $v$ in $V(G)$, $\sum_{u \in V(G)}m(v,u)=d$, which is the number of edges incident to $v$.

\begin{thm}\label{O1}
If $G$ is a $d$-regular multigraph with $\lambda_2(G) < \frac{d-1+\sqrt{9d^2-10d+17}}4$, then $\kappa'(G) \ge 2$.
\end{thm}

\begin{thm}\label{O2}
For $t \ge 2$, if $G$ is a $d$-regular multigraph with $\lambda_2(G) < d-t$,
then $\kappa'(G) \ge t+1$. Furthermore, if $t$ is odd and $G$ is a $d$-regular multigraph with $\lambda_2(G)<d-t+1$, then $\kappa'(G) \ge t+1$.
\end{thm}

Note that $\lambda_2(G) < \frac{d-1+\sqrt{9d^2-10d+17}}4$ in Theorem~\ref{O1} and $\lambda_2(G) < d-t$ in Theorem~\ref{O2} can be replaced by $\mu_2(G) > \frac{3d+1-\sqrt{9d^2-10d+17}}{4}$ and $\mu_2(G) > t$, respectively, since $\mu_2(G)=d-\lambda_2(G)$.
Also, in Theorem~\ref{O1} and~\ref{O2}, complete graphs are  the exceptions although the author did not mention in the theorems.

The author also characterized when a $d$-regular multigraph $G$ has $\lambda_2(G) = \theta(d,t)$
and $\kappa'(G)=t$ for every positive integer $t \in [d-1]$,
where 
$$\theta(d,t)=
\begin{cases}
\frac{d-1+\sqrt{9d^2-10d+17}}4 \qquad \text{ if } t=1 \\
d-t \qquad\qquad\qquad ~~\text{ if } t \text{ is even} \\
d-t+1 \qquad \qquad~~~\text{ if } t \text{ is odd}.
\end{cases}
$$


Many researchers investigated relationships between $\lambda_2$ (or $\mu_2$) and $\kappa'$ in a $d$-regular simple graph or a $d$-regular multigraph as seen above. In this paper, we study a relationship between $\lambda_2$  (or $\mu_2$) and $\kappa$ in a $d$-regular multigraph. 
In Section 2, we provide an upper bound for $\lambda_2(G)$ (or a lower bound for $\mu_2(G)$) in a $d$-regular multigraph $G$ to guarantee $\kappa(G) \ge 2$.

\begin{thm}\label{main}
If $G$ is a $d$-regular multigraph with $\lambda_2(G) < \frac34d$, except for the 2-vertex $d$-regular multigraph, then $\kappa(G)\ge 2$.
\end{thm}

In Section~\ref{consec}, for any fixed positive integer $t\ge 2$ and infinitely many $d$, we construct $d$-regular multigraphs $H_{d,t}$ such that $\kappa(H_{d,t})=t$ and $\lambda_2(H_{d,t})=0$ (or $\mu_2(H_{d,t})=d$).
This means that although $\mu_2$ is fixed like $d$, the vertex-connectivity of a $d$-regular graph $G$ might vary.

\begin{thm}{\rm (Interlacing theorem) (See \cite{BH},\cite{GR}, Lemma 1.6 in~\cite{O3})}\label{interlacing}
If $A$ is a real symmetric $n \times n$ matrix and $B$ is a principal submatrix of $A$ with order $m \times m$, then for each $i \in [m]$, the $i$-th largest eigenvalue of $A$ is at least the $i$-th largest eigenvalue of $B$.
\end{thm}

By using the Interlacing Theorem, we observe that for any multigraph $G$, the $\alpha(G)$-th largest eigenvalue of $G$ is non-negative,
where $\alpha(G)$ is the maximum size of a vertex set $S$ of a multigraph $G$ such that no pair of two vertices in $S$ is adjacent.
This observation says that  $\lambda_2(G) \ge 0$ (or $\mu_2(G) \le d$) when $G$ is a $d$-regular multigraph
whose underlying graph is not a complete graph. Thus the existence of $H_{d,t}$ and the observation imply that there is no upper bound for $\lambda_2(G)$ (or lower bound for $\mu_2(G)$) in a $d$-regular multigraph $G$ to guarantee a certain vertex-connectivity greater than equal to 3.

For undefined terms, see West~\cite{W} or Godsil and Royle~\cite{GR}.

\section{How does $\kappa(G)=1$ affect $\lambda_2(G)$?}

In this section, we prove an upper bound for $\lambda_2(G)$ in a $d$-regular multigraph $G$ to guarantee $\kappa(G) \ge 2$.
To prove it, we use the quotient matrix method. Consider a partition $V(G) = V_1 \cup \cdots \cup V_s$ of the vertex set of a graph $G$ into $s$ non-empty subsets. For $1 \le  i, j \le s$, let $b_{i,j}$ denote the average number of incident edges in $V_j$ of the vertices in $V_i$. 
The {\it quotient matrix} of this partition is the $s \times s$ matrix whose $(i, j)$-th entry equals $b_{i,j}$.

\begin{thm}\label{quot} {\rm (See Corollary 2.5.4 in~\cite{BH}, Lemma 9.6.1 in~\cite{GR})}
The eigenvalues of the quotient matrix interlace
the eigenvalues of $G$. 
\end{thm}

This partition is {\it equitable} if for each $1 \le  i, j \le s$, any vertex $v \in V_i$ has exactly $b_{i,j}$ neighbors in $V_j$. In this case, the eigenvalues of the quotient matrix are eigenvalues of $G$ and the spectral radius of the quotient matrix equals the spectral radius of $G$ (see \cite{BH, GR} for more details).

Theorem~\ref{quot} is the main tool to prove Theorem~\ref{main}. 

\begin{proof}[Proof of Theorem~\ref{main}]
Assume to the contrary that $\kappa(G) \le 1$.
If $G$ is disconnected, then $\lambda_2(G)=d \ge \frac{3d}4$.
Now, suppose that $\kappa(G)=1$.
Let $v$ be a vertex such that $G-v$ is disconnected.
Let $G_1$ be a component of $G-v$ such that $|N(v)\cap V(G_1)| \le \frac d2$. Let $G_2=G-v-V(G_1)$. Since $G$ is $d$-regular,
we have $|N(v)\cap V(G_2)| \ge \frac d2$.
Let $n_1=|V(G_1)|, n_2=|V(G_2)|, m_1=|V(G_1)\cap N(v)|,$ and
$m_2=|V(G_2)\cap N(v)|$. Note that both $n_1$ and  $n_2$ are positive since  $G$ is not the 2-vertex $d$-regular multigraph. Consider the vertex partition $\{V(G_1), V(G_2)\cup v\}$. Let $Q$ be the quotient matrix of the partition.
Then we have 

$$Q=\begin{bmatrix}
d - \frac{m_1}{n_1} & \frac{m_1}{n_1} \\
\frac{m_1}{n_2+1} &  d - \frac{m_1}{n_2+1}
\end{bmatrix},$$

\noindent
and its characteristic polynomial is $$\left(x-d+\frac{m_1}{n_1}\right)\left(x-d+\frac{m_1}{n_2+1}\right) - \frac{m_1}{n_1}\frac{m_1}{n_2+1}
=(x-d)\left(x-d+\frac{m_1}{n_1}+\frac{m_1}{n_2+1}\right).$$

Thus by Theorem~\ref{quot}, 
\begin{equation}\label{equ1}
\lambda_2(G) \ge \lambda_2(Q) = d-\frac{m_1}{n_1}-\frac{m_1}{n_2+1}.
\end{equation}

Consider the vertex partition $\{V(G_1), v, V(G_2)\}$. Let $Q'$ be the quotient matrix of the partition.
Then we have 

$$Q'=\begin{bmatrix}
d - \frac{m_1}{n_1} & \frac{m_1}{n_1} & 0\\
m_1 & 0 & m_2 \\
0 & \frac{m_2}{n_2} &  d - \frac{m_2}{n_2}
\end{bmatrix},$$

\noindent
and its characteristic polynomial is $$(x-d)\left[ (x-d)^2+\left(\frac{m_1}{n_1}+\frac{m_2}{n_2}+ m_1+m_2\right)(x-d) +
m_1m_2\left(\frac 1{n_1}+\frac 1{n_2} + \frac 1{n_1n_2} \right) \right].$$

Thus by Theorem~\ref{quot}, 
\begin{equation}\label{equ2}
\lambda_2(G) \ge \lambda_2(Q') = \frac{2d-\left(\frac{m_1}{n_1}+\frac{m_2}{n_2}+ d\right) + \sqrt{\left(\frac{m_1}{n_1}+\frac{m_2}{n_2}+ d\right)^2 - 4m_1m_2\left(\frac {n_1+n_2+1}{n_1n_2} \right)}}{2}.
\end{equation}

Now, we consider two cases to prove depending on the orders of $n_1$ and $n_2$.\\

\noindent
{\it Case 1: $n_1\ge n_2$.} \\

{\it Case 1.1: $n_1 \ge 4$ and $n_2 \ge 3$.}
Since $m_1 \le \frac d2$, by inequality~(\ref{equ1}),
$$\lambda_2(G) \ge d - \frac d{8} - \frac d{8} = \frac {3d}4.$$
By plugging the numbers $n_1=4, n_2=3,$ and $m_1=\frac d2$, into inequality~(\ref{equ2}), we have 
$$\lambda_2(G) > \frac {3d}4.$$

{\it Case 1.2: $n_1 = 3$ and $n_2 = 3$.} By plugging $n_1=3$ and $n_2=3$ into inequality~(\ref{equ2}), we have 
$$\lambda_2(G) \ge \frac{2d + 3 \sqrt{(\frac{4d}3)^2 - 4m_1m_2(\frac 79)}}{6} \ge \frac{5d}6> \frac{3d}4,$$
since $m_1m_2 \le \frac{d^2}{4}$.

{\it Case 1.3: $n_1 \ge 6$ and $n_2 = 2$.}
Since $m_1 \le \frac d2$, by inequality~(\ref{equ1}),
$$\lambda_2(G) \ge d - \frac d{12} - \frac d{6} = \frac {3d}4.$$
By plugging the numbers $n_1=6, n_2=2,$ and $m_1=\frac d2$, into inequality~(\ref{equ2}), we have $$\lambda_2(G) > \frac {3d}4.$$

~~~{\it Case 1.3.1: $3\le n_1\le 5$.}  By plugging $n_2=2$ into inequality~(\ref{equ2}), we have  $$\lambda_2(G) \ge \frac{d + \frac{(n_1-2)m_1}{2n_1} + 2\sqrt{[\frac{3d}2 -\frac{(n_1-2)m_1}{2n_1}]^2 - 4m_1m_2(\frac {n_1+3}{2n_1})}}{4} $$$$> \frac{d + 2\sqrt{\frac{9d^2}4 - \frac{3(n_1-2)m_1d}{2n_1} - (\frac {n_1+3}{2n_1})d^2}}{4} \ge \frac{d + 2\sqrt{\frac{9d^2}4 - \frac{3(n_1-2)d^2}{4n_1} - (\frac {n_1+3}{2n_1})d^2}}{4} = \frac{3d}4,$$ since $m_1 \le \frac d2$.

~~~{\it Case 1.3.2: $n_1=2$.} By plugging $n_1=2$ and $n_2=2$ into inequality~(\ref{equ2}), we have 
$$\lambda_2(G) \ge \frac{d + 2\sqrt{(\frac{3d}2)^2 - 4m_1m_2(\frac 54)}}{4} \ge \frac{3d}4.$$

In Case 1, equality holds only when $n_1=2$, $n_2=2$, and $m_1=m_2=\frac d2$.

\noindent
{\it Case 2: $n_1 < n_2$.}

{\it Case 2.1: $n_1 \ge 3$ and $n_2 \ge 5$.} Since $m_1 \le \frac d2$, by inequality~(\ref{equ1}),
$$\lambda_2(G) \ge d - \frac d{6} - \frac d{12} = \frac {3d}4.$$
By plugging the numbers $n_1=3, n_2=5,$ and $m_1=\frac d2$ into inequality~(\ref{equ2}), we have $$\lambda_2(G) > \frac{3d}4.$$

{\it Case 2.2: $n_1=3$ and $n_2=4$.}
By plugging $n_1=3$ and $n_2=4$ into inequality~(\ref{equ2}), 

$$\lambda_2(G) \ge  \frac{2d+\frac{m_2}{4} + 3 \sqrt{(\frac{m_1}{3}+\frac{m_2}4 + d)^2 - 4m_1m_2 (\frac 23)}}{6} > \frac{2d+\frac d{8}+3\sqrt{\frac{25d^2}{16}-\frac{2d^2}3}}6 \ge \frac {3d}4,$$
since $m_2 \ge \frac d2$ and $\frac{m_1}3 + \frac{m_2}4 + d =  \frac{m_1}{12} + \frac{5d}4$.

{\it Case 2.3: $n_1=2$ and $n_2\ge 3$.}
By plugging $n_1=2$ into inequality~(\ref{equ2}), 

$$\lambda_2(G) \ge  \frac{d+\frac{(n_2-2)m_2}{n_2} + 2 \sqrt{(\frac{m_1}{2}+\frac{m_2}{n_2} + d)^2 - 4m_1m_2 \left(\frac{3+n_2}{2n_2}\right)}}{4} $$
$$> \frac{d+\frac{(n_2-2)d}{n_2}+2\sqrt{\left(\frac{(n_2+1)d}{n_2}\right)^2-\frac{(3+n_2)d^2}{2n_2}}}4 \ge \frac{d+\frac{(n_2-2)d}{n_2}+2d\sqrt{\frac{n_2^2+n_2+2}{2n_2^2}}}4 > \frac {3d}4,$$
since $\frac{m_1}2 + \frac{m_2}{n_2} + d =  \frac{(n_2-2)m_1}{2n_2} + \frac{(n_2+1)d}{n_2}$ and 
$$1+\frac{n_2-2}{n_2}+2\sqrt{\frac{n_2^2+n_2+2}{2n_2^2}} = \frac{2n_2-2+\sqrt{2n_2^2+2n_2+4}}{n_2} > 3.$$
\end{proof}

Since $\kappa(G)=\kappa'(G)$ for $3$-regular multigraph $G$, we can have a better bound for a 3-regular multigraph in Theorem 1.4 (see Theorem~\ref{O1}).

\section{Sharp Examples, and $\kappa(G) \ge 2$ may not affect $\lambda_2(G)$} \label{consec}
In this section, by presenting a $d$-regular multigraph $H_{d,1}$ with $\kappa(H_{d,1})=1$ and   $\lambda_2(H_{d,1})=\frac34d$, we show that the upper bound for $\lambda_2$ in Theorem~\ref{main} is sharp.
Furthermore,  for every positive integer $t \ge 2$, we construct $d$-regular multigraphs $H_{d,t}$ with $\kappa(H_{d,t})=t$, $m(H_{d,t})=\frac dt$, and  $\lambda_2(H_{d,t})=0$. These graphs show that inequality~(\ref{ineq1}) in Theorem~\ref{fiedlermultiversion} is sharp.
Furthermore, by the observation that if $G$ is a multigraph whose underlying graph is not a complete graph,
then  $\lambda_2(G) \ge 0$, the graphs $H_{d,t}$ guarantee that the bound cannot be improved although the vertex-connectivity of a $d$-regular multigraph varies.


\begin{obs} \label{obs}
For a positive integer $d \ge 3$,
there exists only one multigraph with three vertices such that every vertex has degree $d$, except only one vertex with degree $\frac d2$.
\end{obs}
\begin{proof}
Let $G$ be a multigraph with three vertices, say $v_1, v_2$, and $v_3$ such that $v_1$ and $v_2$ have degree $d$,
and $v_3$ has degree $\frac d2$. Let $a, b, $ and $c$ be the number of edges between $v_1$ and $v_2$,
between $v_2$ and $v_3$, and between $v_3$ and $v_1$, respectively.
Since $v_1$ and $v_2$ have degree $d$, and $v_3$ has degree $\frac d2$, we have $a+c=a+b=d$, $b+c=\frac d2$, and $2(a+b+c)=\frac{5d}2$. Thus, we have $a=\frac {3d}4$ and $b=c=\frac d4$,
which gives the desired result.
\end{proof}

\begin{figure}
\begin{center}
\includegraphics[height=3cm]{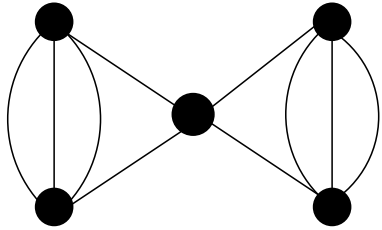}
\end{center}
\caption{The $4$-regular multigraph with $\lambda_2= 3$ and $\kappa=1$}
\end{figure}

\begin{figure}
\begin{center}
\includegraphics[height=5cm]{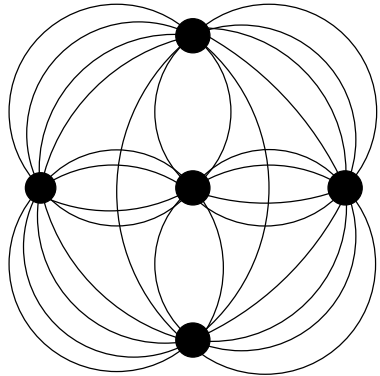}
\end{center}
\caption{The $12$-regular multigraph with $\lambda_2= 0$ and $\kappa=3$}
\end{figure}

\begin{cons}\label{construction}
{\rm Let $B_{d,1}$ be the 3-vertex multigraph guaranteed by Observation~\ref{obs}.
Let $H_{d,1}$ be the graph obtained from two copies of $B_{d,1}$ by identifying the vertices with degree $\frac d2$. 
Note that $H_{d,1}$ is the smallest $d$-regular multigraph 
with $\kappa(H_{d,1})=1$.

Let $t \ge 2$ be fixed. Suppose that for some non-negative integer $k$, $d \ge 3$ is a positive integer such that $\frac{(t-2)d}{t(t-1)}=k$.  For any positive integer $a$, we can have infinitely many $d$ satisfying the equality by taking $d=at(t-1)$.
Let $C_{d,t}$ be the graph obtained from the complete graph $K_{t}$ by duplicating each edge $\frac{(t-2)d}{t(t-1)}$ times. Note that each vertex in $C_{d,t}$ has degree $\frac{(t-2)d}t$.
Let $H_{d,t}$ be the graph obtained from two vertices $x, y$ and one copy of $C_{d,t}$ by adding $\frac dt$ edges between
two $x$ and each vertex in $C_{d,t}$ and between $y$ and each vertex in $C_{d,t}$. Note that $H_{d,t}$ is a  $d$-regular multigraph with $\kappa(H_{d,t})=t$ and $m(H_{d,t})=\frac dt$.
}
\end{cons}

See Figure 1 for $d=4$ and $t=1$, and Figure 2 for $d=12$ and $t=3$.

\begin{lem}\label{main1}
$\lambda_2(H_{d,1})=\frac 34d$ and $\lambda_2(H_{d,t})=0$.
\end{lem}

\begin{proof}
Let $A$ be the adjacency matrix of $H_{d,1}$.
Then we have 

$$A=\begin{bmatrix}
0 & \frac{3d}{4} & \frac{d}{4} & 0 & 0 \\
\frac{3d}{4} & 0 & \frac{d}{4} & 0 & 0 \\
\frac{d}{4} & \frac{d}{4} & 0 & \frac{d}{4} &\frac{d}{4} \\
0 & 0 & \frac{d}{4} & 0 & \frac{3d}{4} \\
0 & 0& \frac{d}{4} & \frac{3d}{4} & 0
\end{bmatrix},$$

\noindent
and its characteristic polynomial is $(x-d)(x-\frac 34d)(x+\frac d4)(x+\frac{3d}{4})^2,$ which means $\lambda_2(H_{d,1})=\frac{3d}{4}$.

For $\lambda_2(H_{d,t})=0$, partition $V(H_{d,t})$ into two parts; two single vertices and $V(C_{d,t})$. The quotient matrix of $(H_{d,t})$ is 

$$Q=\begin{bmatrix}
0 & d \\
\frac{2d}{t} & \frac{(t-2)d}{t}
\end{bmatrix}.$$
The characteristic polynomial of the matrix is $p(x)=(x-d)(x+\frac{2d}{t})$. Since the vertex partition is equitable, $H_{d,t}$ has roots $d, -\frac{2d}t$ by the fact mentioned in the paragraph after Theorem~\ref{quot}.

Consider the subspace $W \subset \mathbb{R}^{t+2}$ of vectors which are constant on each part of the two part equitable partition. The lifted eigenvectors corresponding to the two roots of $p(x)$ form a basis for $W$. The remaining eigenvectors in a basis of eigenvectors for $H_t$ can be chosen to be perpendicular to the vectors in $W$. Thus, they may be chosen to be perpendicular to the characteristic vectors of the parts in the two part equitable partition since these characteristic vectors form a basis for $W$. This implies that these eigenvectors will correspond to the non-trivial eigenvalues of the graph obtained as a disjoint union of $2K_1$ and $C_t$. The corresponding eigenvalue will be  $-\frac{(t-2)d}{t(t-1)}$ with multiplicity $(t-1)$. These $(t-1)$ eigenvalues together with the three roots of $p(x)$ form the spectrum of the graph $H_t$. Thus $\lambda_2(H_{d,t})=0$.
\end{proof}

\section{Questions}
In this paper, for fixed positive integer $t$, we proved best upper bounds for $\lambda_2(G)$ in a $d$-regular multigraph $G$ with $\kappa(G) \ge t$. If we fix the number of vertices, then what can we say about an upper bound for $\lambda_2$? The bound should be expressed in terms of not only the degree, but also the number of vertices.

\begin{que}
For a positive integer $t$, what is the best upper bound for $\lambda_2(G)$ in an $n$-vertex $d$-regular multigraph $G$ with $x \ge t+1$, where $x=\kappa(G)$ or $\kappa'(G)$?
\end{que}


We noted that a simple non-complete multigraph has the non-negative second largest eigenvalue. Since $\lambda_2(K_n)=-1$ for all $n \ge 2$, we have that
a simple graph $G$ has a negative $\lambda_2(G)$
if and only if $G$ is a complete graph with at least two vertices. More generally,
a simple graph $G$ has only one positive eigenvalue
if and only if $G$ is a complete multipartite simple graph (See~\cite{S}).

\begin{que}
What is a necessary and sufficient condition for a multigraph $G$ to have only one positive eigenvalue?
\end{que}

Note that the graph $H_{d,t}$ is a complete multipartite graph with $\lambda_2(H_{d,t})=0$. However, there are complete multipartite multigraphs with a positive second largest eigenvalue. For example, complete graphs with large multiplicity are such graphs. There are many graphs, which are not graphs obtained from a complete graph by duplicating edges. Here is one example. Let $F_4$ be the graph obtained from complete graph on 4 vertices by duplicating two non-incident edges $\frac d2 -2$ times. Note that each vertex in $H_4$ has degree $\frac d2$. Now, let $G_4$ be the graph obtained from $F_4$ and two vertices $x, y$ by adding $\frac d4$ edges between $x,y$
and each vertex in $F_4$. Then $G_4$ is a $d$-regular multigraph with eigenvalues $d, -4+\frac 12d, 0, (-\frac d2 +2)^{(2)}, -\frac d2$, where the exponent is the multiplicity. For $d \ge 8$, we have $\lambda_2(G_4)=-4+\frac 12d$. Note that $G_4$ is a complete multipartite multigraph with a positive second largest eigenvalue.
Thus we need a different family of graphs for such a condition.

\section*{Acknowledgment}
This work was supported by the National Research Foundation of Korea(KRF) grant funded by the Korea government(MSIP) (no. 2016R1A5A1008055).

\end{document}